\documentclass[a4paper,english,reqno]{amsart}

\usepackage{amssymb}
\usepackage{amsthm}
\usepackage{amsmath}
\usepackage{bbm}
\usepackage[all]{xy}
\usepackage{graphicx}
\usepackage{subfigure}
\usepackage{microtype}
\usepackage{booktabs}
\usepackage{enumerate}
\usepackage{pgf}
\usepackage[latin1]{inputenc}
\usepackage{verbatim}
\usepackage{cite}
\usepackage{pgf,tikz}
\usepackage{url}
\usepackage{xcolor}
\usetikzlibrary{arrows}
\definecolor{xdxdff}{rgb}{0.66,0.66,0.66}
\definecolor{zzzzzz}{rgb}{0.6,0.6,0.6}

\xyoption{arc}
\usetikzlibrary{arrows,shapes,calc,positioning}
\usepackage{manfnt}

\newcommand{\cF}{\mathcal{F}}
\newcommand{\R}{\mathbb{R}}
\newcommand{\Z}{\mathbb{Z}}
\newcommand{\N}{\mathbb{N}}

\renewcommand{\H}{\mathbb{H}}

\newtheorem{theorem}{Theorem}

\newtheorem{proposition}{Proposition}
\newtheorem{corollary}{Corollary}

\theoremstyle{definition}
\newtheorem{definition}{Definition}
\newtheorem{example}{Example}
\newtheorem{remark}{Remark}

\newtheoremstyle{named}{}{}{\itshape}{}{\bfseries}{.}{.5em}{\thmnote{#3} #1}
\theoremstyle{named}
\newtheorem*{namedlemma}{Lemma}

\newtheorem*{property}{Property}

\begin{document}  

    \title[Minimality of the horocycle flow]{Minimality of the horocycle flow on  laminations by hyperbolic
    surfaces with non-trivial topology}

%
%
%
%
%
%

\author[F. Alcalde]{Fernando Alcalde Cuesta$^{1,2}$}
\address{$^1$~GeoDynApp - ECSING Group, Spain. \vspace*{-2ex}}
\address{$^2$~Instituto de Matem\'aticas, Universidade de Santiago de Compostela, E-15782, Santiago de Compostela, Spain. }
\email{fernando.alcalde@usc.es}

\author[F. Dal'Bo]{Fran\c{c}oise Dal'Bo$^3$}
\address{$^3$~Institut de Recherche Math\'ematiques de Rennes,
 Universit\'e de  Rennes 1, F-35042 Rennes, France.}
\email{francoise.dalbo@univ-rennes1.fr}

\author[M. Mart\'{\i}nez]{Matilde Mart\'{\i}nez$^4$}
\address{$^4$~Instituto de Matem\'atica y Estad\'{\i}stica Rafael Laguardia, Facultad de Ingenier\'{\i}a,
Universidad de la Rep\'ublica, J.Herrera y Reissig 565, C.P.11300 Montevideo, Uruguay.}
\email{matildem@fing.edu.uy}

\author[A. Verjovsky]{Alberto Verjovsky$^5$}
\address{$^5$~Universidad Nacional Aut\'onoma de M\'exico,
Apartado Postal 273, Admon. de correos \#3, C.P. 62251 Cuernavaca,
Morelos, Mexico.}
\email{alberto@matcuer.unam.mx}

\date{}
\thanks{}


\subjclass[2010]{}

\begin{abstract}
We consider a minimal compact lamination by hyperbolic surfaces. We prove that if 
 no leaf is simply connected,
then the horocycle flow on its unitary tangent bundle is minimal.
\end{abstract}

\maketitle

\section{Introduction} \label{Sintro}

The geodesic and horocycle flows on hyperbolic surfaces are two classical examples of flows in homogeneous
spaces. Their dynamical
and ergodic properties were studied in the 1930s by E. Hopf and G.A. Hedlund.

A compact hyperbolic
surface $S$ is the quotient of
the Poincar\'e upper-half plane $\H$ by a cocompact
torsion-free Fuchsian group $\Gamma$. Its unit tangent bundle
$T^1S$, that can be seen as the
quotient of $PSL(2,\R)$ by $\Gamma$, is the phase space of these flows. The geodesic flow is uniformly
hyperbolic on $T^1S$, and its stable manifolds are precisely the horocyclic orbits. Alternatively, these
flows can be seen as coming from  the diagonal and
upper-triangular unipotent one-parameter subgroups of $PSL(2,\R)$ acting
on $\Gamma\backslash PSL(2,\R)$.

In this context, Hedlund proved in 1936 that the horocycle flow is minimal -- that is, all its orbits are
dense \cite{Hedlund}.

\medskip

M. Ratner completed in the 1990s the ergodic-theoretical and topological descriptions of the dynamics of
unipotent groups on homogeneous spaces, with a conclusive theorem in a subject that had seen important
contributions by S.G. Dani, H. Furstenberg, G.A. Margulis, J. Smillie, among others.

Hedlund's theorem has a converse (see for example \cite{Dal'Bo2}): If the horocycle flow on a hyperbolic
surface $S$ is minimal, then $S$ must be compact.
In fact the dynamics of the horocycle flow is also
well understood for {\em geometrically finite} surfaces in general
(i.e. those whose fundamental group is finitely generated):
in the
restriction to the non wandering set of the flow, the orbits are dense or periodic (see for example
\cite{Dal'Bo}).

The topological behavior of horocycles poses
many questions in the case where the surface is geometrically infinite (i.e.
its fundamental group is not finitely generated). M. Kulikov has constructed in \cite{Kulikov} a
{\em geometrically infinite} surface without minimal sets for the horocycle flow and
S. Matsumoto has announced
other examples in \cite{Matsumoto} and revisited them in \cite{Matsumoto2}. The dynamics of the recurrent
horocycle orbits under different assumptions has been extensively studied by  Y. Coud\`ene, F. Ledrappier,
F. Maucourant and B. Schapira, among others.

\medskip

From the point of view of ergodic theory, in the infinite-volume setting, A.N. Starkov conjectured that
this flow is ergodic
(with respect to Liouville measure) if and only if the action of $\Gamma$ on the boundary $\partial\H$ is
ergodic (with respect to Lebesgue measure). By duality, this is equivalent to the ergodicity of the joint
action of the geodesic and horocycle flows on $T^1 S$, which is  just the action of the affine group $B$.
This result was proved in a special case by M. Babillot and F. Ledrappier  and independently by M. Pollicott,
and in general by V. Kaimanovich. For these and other related results see for example
\cite{Starkov} and \cite{Kaimanovich}. See also \cite{Roblin} for a complete survey in the more general
setting of pinched Hadamard manifolds.

\medskip

In the present paper we consider a compact lamination $(M,\cF)$ by hyperbolic surfaces. 
Its unit tangent
bundle is a lamination
obtained by taking the unit tangent bundles of the leaves of $\cF$, and it is defined by a 
continuous
$PSL(2,\R)$-action. We study the horocycle flow --that is, the action of the
upper triangular unipotent subgroup $U$ of $PSL(2,\R)$-- on this space when the lamination 
is minimal.
The actions of the diagonal group $D$ and the affine group $B$ play an important role in this 
study.
The idea of studying the geodesic flow for compact laminations is not new. Its dynamics 
and its relation
with the dynamics of the lamination $\cF$ have been studied in several contexts, see for 
example
\cite{Bonatti-Gomez-Mont-Vila}, \cite{Hurder}.

As the lamination is minimal,
one might expect that the compactness of the ambient space forces the minimality of the horocycle flow,
like for compact hyperbolic surfaces.
Nevertheless, this is not the case, since there are examples where not even the $B$-action is minimal
(see \cite{Fernando-Francoise} and \cite{Alberto-Matilde}). The last two
authors have posed the following question some years ago: {\em Is the minimality of the $B$-action 
equivalent to the
minimality of the horocycle flow?} 

\medskip

The main result of this paper gives a positive answer to the question above for a large family of
laminations by hyperbolic surfaces:

\begin{theorem}
\label{main_theorem}

Let $(M,\cF)$ be a minimal compact lamination by hyperbolic surfaces. If $\cF$ 
has no simply connected leaves,
then the horocycle flow on the unitary tangent bundle
$ T^1 \cF$ is minimal.
\end{theorem}

Minimality of the foliated horocycle flow has already been proved in an algebraic setting. 
More precisely, in {\it Remarks on the dynamics of the horocycle flow for homogeneous 
foliations by hyperbolic surfaces} (\cite{Fernando-Francoise}), 
the authors 
prove the minimality 
of the horocyclic flow for the homogeneous case using algebraic methods. On the other hand, in the 
work {\it Horocycle flows for laminations by hyperbolic Riemann surfaces and Hedlund's theorem} 
(\cite{Alberto-Matilde}), the authors consider a general compact minimal 
lamination by hyperbolic surfaces and they mostly study, with topological tools, the action of the affine 
group generated by the joint action of the horocycle and geodesic flows. Its main contribution when it 
comes to the subject of the horocycle flow is to prove that under certain conditions it is 
topologically transitive, and to present examples where different types of dynamical 
behaviour appear, which in turn leads to the conjecture which is addressed in the present paper.

\bigbreak

A foliation by hyperbolic surfaces that satisfies the hypotheses of Theorem \ref{main_theorem} is the
{\em Hirsch foliation},
briefly described below.
A more detailed description of it can be found in \cite[Page 371]{Candel-Conlon}.
It was this foliation that motivated many of the ideas present in this paper. Furthermore,
we believe that in this example, because of the simplicity of its construction, many of the arguments we
use become particularly transparent.

\begin{example}[The Hirsch foliation]

Let $P$ be the closed unit disk minus two open disks of radius   $1/4$ centered at $-1/2$ and $1/2$ in the complex
plane; namely, a {\em pair of pants}.
The set $P$ is then invariant under the involution blue $\sigma : Z \in P \mapsto -Z \in P$. The suspension
$\mathbf{S}$ of $\sigma$ is a non-trivial fiber bundle over $\mathbb{S}^1$, whose holonomy interchanges the
two interior boundary components of the pair of pants. Therefore, it can be obtained from a solid torus
removing from its interior a thinner solid torus wrapped two times around the generator of its fundamental
group. Let $p:\mathbf{S}\to \mathbb{S}^1$ be the fibration. The boundary of $\mathbf{S}$ consists of two
tori $T_1$ and $T_2$.  The restriction of $p$ to the external torus $T_1$ is itself a fibration $p_1$ over $S^1$ with 
fiber $\mathbb{S}^1$, while the restriction of $p$ to the internal torus $T_2$ is obtained from the fibration $p_2$ over $S^1$
with fiber $\mathbb{S}^1$ by composing with the covering map sending 
$z  =e^{2 \pi i \theta} \in S^1$ to $z^2 = e^{2 \pi i 2\theta} \in \mathbb{S}^1$.
(see Figure~\ref{hirsch}).



Let $M$ be the 3-manifold obtained from $\mathbf{S}$ by glueing smoothly $T_1$ and $T_2$ via a
diffeomorphism $h$ that sends  $p_1^{-1}(z)$ to $p_2^{-1}(z)$ for every $z$ in $\mathbb{S}^1$. For example, 
$h(Z,z) = (Zz/4+1/2,z^2)$ where we identify each point $(Z,z) \in \partial P \times S^1$ with its image in $T_1$. The fibration in 
$\mathbf{S}$ projects onto a foliation $\mathcal{F}$ having two types of leaves: there are countably many leaves of genus one and a Cantor
set of ends,  which correspond to periodic points of the doubling map $f(\theta)=2\theta$ (mod $1$), and all other leaves have no genus and a Cantor set of ends (they are so-called {\it Cantor trees}).  All leaves are dense and those of
genus one are exactly the ones with non-trivial holonomy (see Figure~\ref{Hirschwithouthol}-\ref{Hirschwithhol}).


 The Hirsch foliation admits many structures of foliation {\em by hyperbolic surfaces}, as explained in Section
\ref{Sshypmet}. The moduli space of such structures has recently been described in \cite{Alvarez-Lessa}.

\end{example}

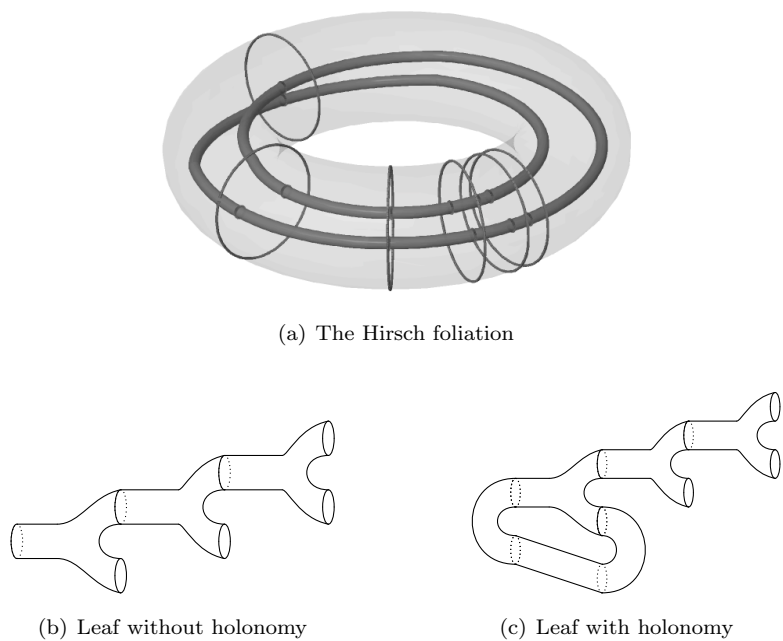
\begin{figure}
		\subfigure[The Hirsch foliation]{
		\begin{tikzpicture}
\clip(-4,-2.2) rectangle (4,2);
\node (A) at (0,0) {\includegraphics[width=7.2cm]{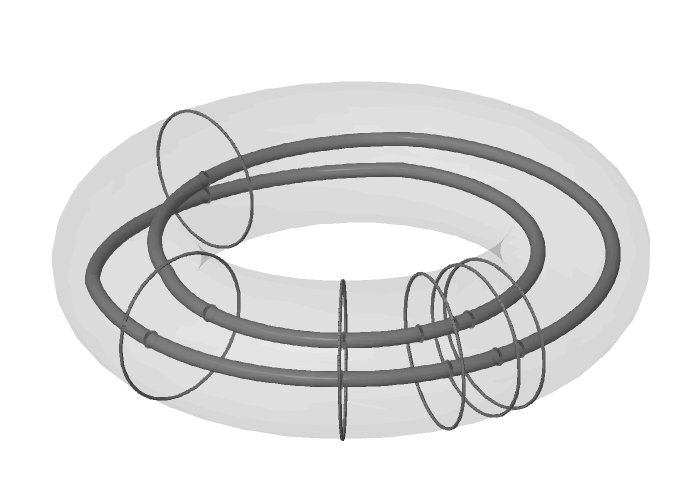}};
\end{tikzpicture}
\label{hirsch}
}
		\subfigure[Leaf without holonomy]{
 \begin{tikzpicture}
\node (B) at (0,0.){
 \rotatebox{-90}{
 \xy
 0;/r.18pc/:
(0,-3)*\ellipse(3,1){.};
(0,-3)*\ellipse(3,1)__,=:a(-180){-};
(-3,6)*\ellipse(3,1){.};
(-3,6)*\ellipse(3,1)__,=:a(-180){-};
(3,6)*\ellipse(3,1){-};
(-6,15)*\ellipse(3,1)__,=:a(-180){-};
(-6,15)*\ellipse(3,1){.};
(-3,12)*{}="1"; 
(3,12)*{}="2"; 
(-9,12)*{}="A2";
(9,12)*{}="B2"; 
"1";"2" **\crv{(-3,7) & (3,7)};
(-3,-6)*{}="A0";
(3,-6)*{}="B0"; 
(-3,1)*{}="A1";
(3,1)*{}="B1"; 
"A0";"A1" **\dir{-};
"B0";"B1" **\dir{-}; 
"B2";"B1" **\crv{(8,7) & (3,5)};
"A2";"A1" **\crv{(-8,7) & (-3,5)};
(0,15)*\ellipse(3,1){-};
(-9,30)*{}="11"; 
(-3,30)*{}="21"; 
(-15,30)*{}="A21";
(3,30)*{}="B21"; 
"11";"21" **\crv{(-9,25) & (-3,25)};
(-9,12)*{}="A01";
(-3,12)*{}="B01"; 
(-9,22)*{}="A11";
(-3,22)*{}="B11"; 
"A01";"A11" **\dir{-};
"B01";"B11" **\dir{-}; 
"B21";"B11" **\crv{(2,25) & (-3,23)};
"A21";"A11" **\crv{(-14,25) & (-9,23)};
(-9,24)*\ellipse(3,1){-};
(-3,24)*\ellipse(3,1){-};
(-15,48)*{}="12"; 
(-9,48)*{}="22"; 
(-21,48)*{}="A22";
(-3,48)*{}="B22"; 
"12";"22" **\crv{(-15,43) & (-9,43)};
(-15,30)*{}="A02";
(-9,30)*{}="B02"; 
(-15,40)*{}="A12";
(-9,40)*{}="B12"; 
"A02";"A12" **\dir{-};
"B02";"B12" **\dir{-}; 
"B22";"B12" **\crv{(-4,43) & (-9,41)};
"A22";"A12" **\crv{(-20,43) & (-15,41)};
\endxy}};
\end{tikzpicture}
\label{Hirschwithouthol}
}
\hspace{1cm}
\subfigure[Leaf with holonomy]{
\begin{tikzpicture}
\node (A) at (0,0 ){
\rotatebox{-90}{
 \xy
 0;/r.15pc/:
(0,-3)*\ellipse(3,1){.};
(0,-3)*\ellipse(3,1)__,=:a(-180){.};
(-3,6)*\ellipse(3,1){.};
(-3,6)*\ellipse(3,1)__,=:a(-180){-};
(3,6)*\ellipse(3,1){.};
(3,6)*\ellipse(3,1)__,=:a(-180){-};
(-3,12)*{}="1"; 
(3,12)*{}="2"; 
(-9,12)*{}="A2";
(9,12)*{}="B2"; 
"1";"2" **\crv{(-3,7) & (3,7)};
(-3,-6)*{}="A0";
(3,-6)*{}="B0"; 
(-3,1)*{}="A1";
(3,1)*{}="B1"; 
"A0";"A1" **\dir{-};
"B0";"B1" **\dir{-}; 
"B2";"B1" **\crv{(8,7) & (3,5)};
"A2";"A1" **\crv{(-8,7) & (-3,5)};
(3,6)*\ellipse(3,1){.};
(9,6)*\ellipse(3,1){.};
(9,12)*{}="X1"; 
(15,12)*{}="X2";
"X1";"X2" **\crv{(9,16) & (15,16)};
(3,12)*{}="Y1";
(21,12)*{}="Y2";
"Y1";"Y2" **\crv{(3,24) & (21,24)};
(0,-3)*\ellipse(3,1){.};
(6,-3)*\ellipse(3,1){.};
(3,-6)*{}="Z1";
(9,-6)*{}="Z2";
"Z1";"Z2" **\crv{(3,-11) & (9,-11)};
(-3,-6)*{}="W1";
(15,-6)*{}="W2";
"W1";"W2" **\crv{(-3,-18) & (15,-18)};
(6,-3)*\ellipse(3,1){.};
"Z2";"X2" **\dir{-};
"W2";"Y2" **\dir{-};
%
(-6,15)*\ellipse(3,1){.};
(-6,15)*\ellipse(3,1)__,=:a(-180){-};
(0,15)*\ellipse(3,1){-};
(-9,30)*{}="11"; 
(-3,30)*{}="21"; 
(-15,30)*{}="A21";
(3,30)*{}="B21"; 
"11";"21" **\crv{(-9,25) & (-3,25)};
(-9,12)*{}="A01";
(-3,12)*{}="B01"; 
(-9,22)*{}="A11";
(-3,22)*{}="B11"; 
"A01";"A11" **\dir{-};
"B01";"B11" **\dir{-}; 
"B21";"B11" **\crv{(2,25) & (-3,23)};
"A21";"A11" **\crv{(-14,25) & (-9,23)};
(-9,24)*\ellipse(3,1){-};
(-3,24)*\ellipse(3,1){-};
(-15,48)*{}="12"; 
(-9,48)*{}="22"; 
(-21,48)*{}="A22";
(-3,48)*{}="B22"; 
"12";"22" **\crv{(-15,43) & (-9,43)};
(-15,30)*{}="A02";
(-9,30)*{}="B02"; 
(-15,40)*{}="A12";
(-9,40)*{}="B12"; 
"A02";"A12" **\dir{-};
"B02";"B12" **\dir{-}; 
"B22";"B12" **\crv{(-4,43) & (-9,41)};
"A22";"A12" **\crv{(-20,43) & (-15,41)};
\endxy}};

 \end{tikzpicture}
 \label{Hirschwithhol}
 }
 \caption{The Hirsch foliation and its leaves}
 \end{figure}

The non-homogeneous Lie foliations constructed by G. Hector, Matsumoto and G. Meigniez in
\cite{Hector-Matsumoto-Meigniez} are other examples of minimal foliations satisfying the hypotheses
of Theorem \ref{main_theorem} for which the horocycle flow is always minimal.

\medskip

This paper has two main parts. The first one, contained in Sections \ref{Sgood} and \ref{Skey}, 
prove a statement that
is apparently weaker than Theorem \ref{main_theorem}  (see  Theorem
\ref{main_theorem_toy_version} in Section \ref{Sgood}). Basically,   this result states that 
for laminations which are similar to the Hirsch foliation
the horocycle flow is minimal. The second one,
contained in Section \ref{Sfol}, which is purely topological and is unrelated to
the horocycle flow, says that, roughly speaking, 
 laminations without simply connected 
leaves are similar to the Hirsch foliation.
More precisely, its main result is the following:

\begin{theorem}\label{reduction_theorem}
Let $(M,\cF)$ be a minimal compact lamination by hyperbolic surfaces.
The following conditions are equivalent:

\begin{list}{\labelitemi}{\leftmargin=5pt}
 \item[(i)\;\;] No leaf is simply connected.
 \item[(ii)\;] There is a leaf that has an essential loop with trivial holonomy.
 \item[(iii)] There is a leaf which is geometrically infinite.
 \item[(iv)\,] All leaves are coarsely tame.
\end{list}
\end{theorem}

The concept of {\em coarsely tame} hyperbolic surfaces will be defined in Section \ref{Spre}.
Coarsely tame surfaces are a special kind of geometrically infinite hyperbolic surfaces.

Together, Theorems \ref{reduction_theorem} and \ref{main_theorem_toy_version} prove
Theorem \ref{main_theorem}. Furthermore, Theorem \ref{reduction_theorem} tells us precisely
for which  laminations the topological dynamics of the horocycle flow is still unclear. In fact,
rephrasing Theorem \ref{reduction_theorem} we get:

\begin{corollary}\label{reduction_theorem_version_2}
Let $(M,\cF)$ be a minimal compact lamination by hyperbolic surfaces.
The following conditions are equivalent:

\begin{list}{\labelitemi}{\leftmargin=5pt}
 \item[(i)\;\;] There is a simply connected leaf.
 \item[(ii)\;] There is a leaf that is geometrically finite.
 \item[(iii)] All leaves are geometrically finite.
 \item[(iv)\,] All essential loops have non-trivial holonomy.
\end{list}
\end{corollary}

For some of these laminations, though, the horocycle flow is known not to be minimal, but neither is
the corresponding $B$-action. Thus, in the geometrically finite case, the question whether the minimality of the $B$-action implies the minimality of the horocycle flow remains open.

\bigbreak

 Theorem \ref{main_theorem} remains valid for any compact lamination by surfaces of variable
negative curvature. Such a example can be obtained, for example, by multiplying the uniformized hyperbolic 
metric by any positive function which is close to $1$ in the $C^2$ topology.
The geometric arguments used in Section~\ref{Skey} can be adapted to any surface of pinched negative curvature whereas 
the uniformization theorem of \cite{Candel} and \cite{Verjovsky} can be used to prove that if no leaf is simply connected, then all the leaves 
satisfy the hypothesis of the Key Lemma.

\subsection*{Acknowledgements}
We thank Shigenori Matsumoto for kindly sending us an unpublished note on the subject of horocycle flows
without minimal sets. We are also grateful to the referees of Discrete and Continuous Dynamical Systems for a very careful reading of the manuscript 
and many valuable suggestions.
This work was partially supported by Spanish Excellence Grants MTM2010-15471 and  MTM2013-46337-C2-2-P,  Galician Grant GPC2015/006 and European Regional Development Fund,
Grupo CSIC 618 (UdelaR, Uruguay) and project PAPIIT-DGAPA IN103914 (UNAM, Mexico).

\section{Preliminaries and notation} \label{Spre}

\subsection*{Hyperbolic surfaces}

A hyperbolic surface $S$ is the quotient of the hyperbolic
plane $\H$ under the left action of a torsion-free discrete subgroup
$\Gamma$ of the group $PSL(2,\R)$ of orientation preserving isometries of $\H$. This group
acts freely and transitively on the unit tangent
bundle $T^1\H$ of $\H$, which means that we can make the identification $T^1S= \Gamma\backslash PSL(2,\R)$.

When $\Gamma$ is of finite type we say that $S$ is {\em geometrically finite}, else it is {\em geometrically infinite}.

\begin{definition}
We say that $S$ is {\em coarsely tame} if it is noncompact and there exists a
constant $b>0$ such that
any geodesic ray in $S$ either stays in a compact region or intersects infinitely many closed geodesics whose lengths
are bounded from above by $b$.
\end{definition}

\begin{remark}\label{remark:geodesic_rays}

\noindent
(i) A coarsely tame surface is geometrically infinite and its fundamental group is a purely
hyperbolic Fuchsian group of the first kind-- that is, its limit set is the whole boundary at
infinity of the hyperbolic plane. In  particular, it has infinite area.
\medskip

\noindent
(ii) If furthermore it has bounded geometry  as defined in \cite[Section 3]{CGT}, 
there exists a lower bound $a>0$ of the injectivity
radius
of $S$, hence
any geodesic ray in $S$ either stays in a compact region or intersects infinitely many closed
geodesics whose lengths
are bounded
between $a$ and $b$.
\end{remark}

For instance, the leaves of the Hirsch foliation are coarsely tame of bounded geometry. Also,
if $\Gamma$ is a cocompact
Fuchsian group and $N$ is a proper normal subgroup of infinite index, then $ N\backslash \H$ is
coarsely tame
of bounded geometry.
(See \cite{Sarig}.)

Let $D$ and $U$ be the diagonal and unipotent subgroups of $PSL(2,\R)$
$$D=\left\{
\left (
\begin{array}{cc}
e^\frac{t}{2} & 0\\
0 & e^{-\frac{t}{2}}
\end{array}
\right ) \,:\,t\in\R
\right \}
\quad \mbox{and} \quad
U=\left\{
\left (
\begin{array}{cc}
1 & s\\
0 & 1
\end{array}
\right ) \,:\,s\in\R
\right \}.
$$
Their right actions define the geodesic flow $g_t$ and the horocycle flow $h_s$ on $T^1S$, respectively. Therefore, the joint action of $g_t$ and
$h_s$ is the action of the affine  group
$$B=\left\{
\left (
\begin{array}{cc}
a & b\\
0 & a^{-1}
\end{array}
\right ) \,: \,a>0,\ b\in\R
\right\}.$$

If $I\subset\R$, the subset $\{g_t(u)\, :\, t\in I\}$ of the orbit of a point $u\in T^1S$ under the geodesic flow is $g_{\!_{I}}\!(u)$, and a similar notation will be used for subsets of horocycle orbits.

The projection $\pi:T^1S\to S$ is the canonical projection that assigns to each vector in $T^1S$ its base point.

\subsection*{Laminations}
A compact {\em lamination} by surfaces $(M,\cF)$ (or simply $\cF$) consists of a compact metrizable space $M$ together
with a family $\{(U_\alpha, \varphi_\alpha)\}$
such that
\begin{enumerate}
\item $\{U_\alpha\}$ is an open covering of $M$,
\item $\varphi_\alpha: U_\alpha \to D\times T$ is a homeomorphism, where $D$ is a disk in $\R^2$ and $T$ is a
 separable locally compact metrizable space, and
\item for $(x,t)\in \varphi_\beta(U_\alpha\cap U_\beta)$,
$\varphi_\alpha \circ \varphi_\beta^{-1}(x,t) = (\lambda^t_{\alpha\beta}(x), \tau_{\alpha\beta}(t))$, where
$\lambda^t_{\alpha\beta}$
is smooth and depends continuously on $t$ in the $C^\infty$ topology.
\end{enumerate}

We will always work in the  leafwise  smooth setting unless otherwise stated, although
$C^3$ leafwise regularity would be enough for all our purposes.

\medskip

Each $U_\alpha$ is called a {\em foliated chart}, a set of the form $\varphi_\alpha^{-1}(\{x\}\times T)$ being its
{\em transversal}.
The sets of the form $\varphi_\alpha^{-1}(D\times\{t\})$, called {\em plaques}, glue together to form maximal
connected surfaces called {\em leaves}.

\medskip
A lamination is said to be {\em minimal} if all its leaves are dense.

\medskip
The tangent bundle of the lamination $\cF$ is the $\R^2$-bundle over $M$ which can be trivialized on each foliated chart
$$
\begin{array}{rcl}
(U_\alpha\cap U_\beta) \times \R^2 &\to& (U_\alpha\cap U_\beta)\times \R^2\\
(p,v)  &\mapsto& (p, d\lambda^t_{\alpha\beta}(p)(v)).
\end{array}
$$

It is itself a (noncompact) lamination, whose leaves are the tangent bundles of the leaves of $\cF$.

\subsection*{Laminations by hyperbolic surfaces} \label{Sshypmet}
In each foliated chart we can endow each plaque with a Riemannian metric, in a continuous way. Glueing these local
metrics with
partitions of unity gives a Riemannian metric on each leaf, which varies continuously in the $C^\infty$ topology as
we move from leaf to leaf. In particular, leaves of a compact lamination always have bounded geometry.

When $(M,\cF)$ is endowed with such a metric, we define the {\em unit tangent bundle} $T^1\cF$ of
$\cF$ as the subset of the tangent bundle containing vectors of unit length. It is a circle bundle over $M$, and it is a
lamination whose leaves are the unit tangent
bundles of the leaves of $\cF$.

Furthermore, we will assume that there is a  Riemannian metric for which all leaves are {\em hyperbolic surfaces}, that is, 
they have constant curvature $-1$. In fact, any given Riemannian metric endows each leaf with a conformal structure, or equivalently, 
with a Riemann surface structure. If all leaves  are uniformized by the disk,
then the uniformization is continuous and leaves become hyperbolic surfaces.
See \cite{Candel} and \cite{Verjovsky}. The existence of these hyperbolic metrics turns out to be a purely
topological condition. It  is equivalent to every leaf having positive volume
entropy. This is independent of the metric because, $M$ being compact, the restriction to leaves of all
Riemannian metrics are quasi-isometric.

\medskip

Once each leaf has a hyperbolic structure, there is a right continuous $PSL(2,\R)$-action on $T^1\cF$ whose orbits are the unit
tangent bundles to the leaves. We will also denote by $g_t$ and $h_s$ the geodesic and horocycle flows on $T^1\cF$ defined by the action of the one-parameter subgroups $D$ and $U$. The natural right $B$-action on $T^1\cF$ combines these two flows.

\section{The horocycle flow on a minimal foliation by coarsely tame leaves} \label{Sgood}

In next sections, we will prove the following (apparently weaker) version of the main theorem which was stated in the
Introduction:

\begin{theorem}
\label{main_theorem_toy_version}

Let $(M,\cF)$ be a minimal compact lamination by hyperbolic coarsely tame surfaces.
Then the horocycle flow on the unitary tangent bundle
$T^1 \cF$ is minimal.
\end{theorem}

The best known example of an object in the hypotheses of this theorem
is probably the Hirsch foliation (described in the introduction). Studying the dynamics of the horocycle
flow in its unit tangent bundle will involve the
understanding of the dynamics of the horocycle flow on a single leaf, and also considerations about the
interplay between the dynamics of the  lamination and that of the horocycle flow on individual leaves.
\medskip

\begin{proposition}
\label{proposition:geo_flow_not_a_suspension}
Let $(M,\cF)$ be a minimal compact lamination such that
\begin{list}{\labelitemi}{\leftmargin=5pt}

\item[(i)\;] the affine action $T^1 \cF \curvearrowleft B$ is minimal,

\item[(ii)] the horocycle flow on $T^1 \cF$  is transitive, i.e. $\exists \, u \in T^1\cF$ such that
$\overline{h_{\R}(u)}= T^1\cF$.

\end{list}

\noindent
Let $\mathcal{M} \neq \emptyset$ be a minimal set for the horocycle flow in $T^1 \cF $, 
 i.e. a non-empty closed $U$-invariant 
subset of  $T^1 \cF $ which is minimal under inclusion.
If there is a real number $t_0 > 0$ such that $g_{t_0}(\mathcal{M}) = \mathcal{M}$, then
$\mathcal{M} = T^1\cF $.
\end{proposition}

\begin{proof}
If $\mathcal{M} \neq T^1 \cF$, then
$$C = \{ \,  t \in \R \, :  \, g_t(\mathcal{M}) \cap \mathcal{M}\neq\emptyset \,  \}
=\{ \,  t \in \R \, : \, g_t(\mathcal{M}) = \mathcal{M} \,  \}$$
is a discrete subgroup of $\R$, so it is either trivial or cyclic.

By hypothesis, $C$ is nontrivial and hence the $B$-invariant set
\[ g_{\!_\R}\!(\mathcal{M})  = \bigcup_{t \in \R} g_t(\mathcal{M}) = \bigcup_{t \in [0,t_0]} g_t(\mathcal{M}) \]
is closed. Then
\[ g_{\!_\R}\!(\mathcal{M}) = T^1\cF \]
by $B$-minimality.

Let $u\in T^1\cF$ be a point with a dense horocycle orbit. There is $t \in [0,t_0]$ such that $g_t(u) \in \mathcal{M}$.
Since $g_t (h_{\R}(u)) = h_{\R}(g_t(u))$ and $\mathcal{M}$ is $h_{\!_\R}$-invariant,
$g_t(\overline{h_{\R}(u})) \subset \mathcal{M}$ and hence $T^1\cF = \mathcal{M}$.
\end{proof}

Therefore, to prove Theorem \ref{main_theorem_toy_version}, we have to see that it verifies the hypotheses of
Proposition~\ref{proposition:geo_flow_not_a_suspension}.

The first condition, the minimality of the affine action, follows from the minimality of $\cF$ and the fact that
the $B$-action restricted to the unitary tangent bundle of each
(coarsely tame) leaf $L=\Gamma\backslash \H$ is
dual to the action of its fundamental group on $PSL(2,\R)/B=\partial\H$, and hence both actions are minimal.

As for the second condition, the transitivity of the horocycle flow, since $\cF$ is minimal,
it is enough to prove that the horocycle flow is transitive
when restricted to some leaf $L$. This is true of all surfaces having a fundamental group of the first kind,
see for example \cite{Dal'Bo}.

\begin{proof}[Proof of Theorem \ref{main_theorem_toy_version}.]
The condition on the minimal set $\mathcal{M}$ is the object of the next section. More precisely, 
in the next section
we will see that if $S$ is a coarsely tame surface of bounded geometry and 
$\mathcal{C}\subset T^1 S$
is closed and invariant under the
horocycle
flow, then either $\mathcal{C}=T^1S$ or there exist a $t_0>0$ such that
$g_{t_0}(\mathcal{C})\cap \mathcal{C}\neq \emptyset$. Taking $\mathcal{M}$ as in
Proposition \ref{proposition:geo_flow_not_a_suspension} and $\mathcal{C}$ to be its intersection with
the unit tangent bundle to a leaf,
we obtain Theorem \ref{main_theorem_toy_version}.
\end{proof}

\section{The horocycle flow on coarsely tame surfaces of bounded geometry} \label{Skey}

Let $S$ be a coarsely tame  hyperbolic surface of bounded geometry.

Recall that the dense horocyclic trajectories on $T^1S$ are characterized as follows
(see \cite[Theorem 3.1]{Dal'Bo}). Let $\tilde u\in T^1\H$, and denote $\tilde u(+\infty)\in\partial\H$ the extremity of the geodesic ray defined by
$\tilde u$. The following conditions are equivalent:
\begin{enumerate}
 \item  Any horodisc centered at $\tilde u(+\infty)$ contains points of the orbit $\Gamma\zeta$ for some
 $\zeta\in\H$ (i.e. $\tilde u(+\infty)$ is horocyclic).
 \item The horocyclic orbit of the projected point in $T^1S$ is dense.
 \end{enumerate}
Using this characterization, we see that $u\in T^1S$ has a dense horocyclic orbit whenever $g_{\R^+}(u)$ accumulates
in $T^1S$.

\begin{namedlemma}[Key]
\label{key_lemma}
Let $S$ be a hyperbolic surface and $u\in T^1S$. If there is a sequence $\{ \alpha_n \}_{n \geq 1}$
of closed geodesics in $S$ of lengths $\ell(\alpha_n)\in [a,b]$, with $b \geq a>0$, and a sequence of times
$\{ t_n \}_{n\geq 1}$
such that $t_n\longrightarrow+\infty$ and $\pi(g_{t_n}(u))\in\alpha_n$, then there exists a
time $t_0>0$ such that
$$g_{t_0}(u)\in \overline{h_{\!_\R}\!(u)}.$$
\end{namedlemma}

\begin{proof}

The universal cover of $S$ is the hyperbolic plane $\H$, and its unit tangent bundle is $T^1\H$. If $\Gamma=\pi_1(S)$,
we write
$S=\Gamma\backslash \H$. We lift $u$ to $\tilde u\in T^1\H$, that defines a geodesic ray $r$.
Since each geodesic $\alpha_n$ is closed on $S$, it lifts to a geodesic $\tilde\alpha_n$ on $\H$ which is the axis  of a
hyperbolic element $\gamma_n\in \Gamma$. Let $\gamma_n^+$ and $\gamma_n^-$ be the attracting and repelling fixed points of $\gamma_n$ on the boundary at infinity
of $\H$ (see Figure~\ref{axes}).

\begin{figure}
\begin{tikzpicture}[line cap=round,line join=round,>=triangle 45,x=1.0cm,y=1.0cm,scale=0.6]
\clip(-9.3,-1.2) rectangle (6.94,8.6);
\draw [shift={(0,5.29)},color=zzzzzz,fill=zzzzzz,fill opacity=0.1] (0,0) -- (90:0.6) arc (90:163.34:0.6) -- cycle;
\draw [shift={(0,3.16)},color=zzzzzz,fill=zzzzzz,fill opacity=0.1] (0,0) -- (90:0.6) arc (90:161.22:0.6) -- cycle;
\draw [dash pattern=on 2pt off 2pt,domain=-9.3:6.94] plot(\x,{(-0-0*\x)/7});
\draw (0,0) -- (0,7.5);
\draw [shift={(-1.5,0)}] plot[domain=0:pi,variable=\t]({1*3.5*cos(\t r)+0*3.5*sin(\t r)},{0*3.5*cos(\t r)+1*3.5*sin(\t r)});
\draw [shift={(-1.5,0)}] plot[domain=0:pi,variable=\t]({1*5.5*cos(\t r)+0*5.5*sin(\t r)},{0*5.5*cos(\t r)+1*5.5*sin(\t r)});
\draw [->] (0,5.29) -- (-1.9,5.86);
\draw [->] (0,3.16) -- (-1.64,3.72);
\draw [->] (0,1.82) -- (0,2.84);
\draw [fill=xdxdff] (0,1.82) circle (2pt);
 \node (A) at (-6.7,-0.5) {$\gamma_{n+1}^+$};
 \node (B) at (-4.9,-0.5) {$\gamma_{n}^+$};
 \node (C) at (2.1,-0.5) {$\gamma_{n}^-$};
  \node (D) at (4.2,-0.5) {$\gamma_{n+1}^-$};
   \node (i) at (-0.4,2.2) {$\tilde{u}$};
   \node (p) at (0.8,3.4) {$r(t_n)$};
    \node (q) at (1.1,5.7) {$r(t_{n+1})$};
\end{tikzpicture}
\caption{The geodesic ray $r$ and the axes $\tilde\alpha_n$}
\label{axes}
\end{figure}
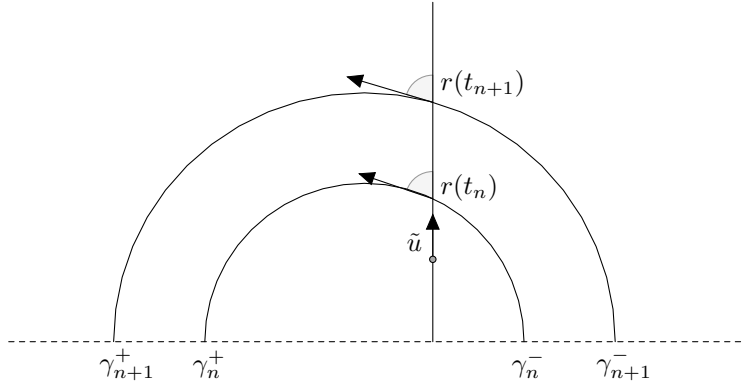


Taking the upper-half-plane model of the hyperbolic plane, we can assume that the geodesic ray directed by $\tilde u$ is
the vertical
half-line $r=\{ ie^t\, :\,t\geq 0\}$, and that the axes $\tilde\alpha_n$  are half-circles orthogonal to the real line.
By hypothesis, they intersect $r$.  Furthermore, we will assume without loss of generality that when we orient
$\tilde\alpha_n$ going from $\gamma_n^-$ to $\gamma_n^+$, the angle between $r$ and $\tilde\alpha_n$ is smaller or equal
than $\pi/2$ for all $n$. Let $r(t_n)$ be the point $ie^{t_n}$ where $r$ intersects
$\tilde\alpha_n$.

We will also denote by $g$ and $h$, respectively, the geodesic and horocycle flows in $T^1\H$.

Proving that there exists $t_0>0$ such that $g_{t_0}(u)\in \overline{h_{\!_\R}\!(u)}$ amounts to finding a sequence
$\{\gamma_n'\}_{n\geq 1}$ in $\Gamma$ and a sequence of times $s_n$ such that
$$\lim_{n\to +\infty}\gamma_n' h_{s_n}(\tilde u)= g_{t_0}(\tilde u).$$

Let $H$ be the horizontal line passing through $i$, that is, the projection to $\H$ of the horocycle orbit of the point $\tilde u$. The above condition simply states that $\gamma_n'(H)$ approaches the horizontal line passing through $ie^{t_0}$ (see Figure~\ref{horapproach}). Therefore, it suffices to prove that there exist a sequence $\{\gamma_n'\}_{n\geq 1}$ in $\Gamma$ and a constant $t_0>0$ such that
\begin{itemize}
\item[(i)~] $\gamma_n'\infty \to \infty$ and
\item[(ii)] $B_{\gamma'_n\infty}(i,\gamma'_ni)\longrightarrow t_0$,
\end{itemize}
where $B$ is the Busemann function given by $B_\xi(x,y)=\lim_{z\to \xi}[d(y,z)-d(x,z)]$.
 Notice that $B(x,y)\leq 0$ if and only if $y$ belongs to  the closed horoball centered at $\xi $ passing through $x$.
\medskip

We will prove (i) and (ii) for $\gamma_n'$ a convenient subsequence of iterates of $\gamma_n$.



We divide the proof in
several steps.

\medskip

\noindent {\em Step 1.} We will show that $\gamma_n\infty\longrightarrow \infty$.
Since $r(t_n)\to \infty$ and it belongs to the axis $\tilde\alpha_n$ of $\gamma_n$, one of the endpoints 
$\gamma_n^+$ or $\gamma_n^-$ of $\tilde\alpha_n$ goes to infinity. If there is a subsequence for which $\gamma_n^+\to\xi\neq\infty$, then 
$\gamma_n^-\to\infty$ and the angle at the point $r(t_n)$ between the ray $r$ and the oriented geodesic 
$\tilde\alpha_n$ would approach $\pi$ as $n$ grew. This is inconsistent with our assumption that this angle 
is bounded above by $\frac{\pi}{2}$. Therefore $\gamma_n^+\to\infty$.

Looking at the dynamics of $\gamma_n\in PSL(2,\R)$ acting on $\H\cup\partial\H$, we see that 
$\gamma_n\infty$ does not belong to the interval with endpoints $\gamma_n^+$  and $\gamma_n^-$, and therefore
$\gamma_n\infty\to\infty$ as stated.
\medskip

\begin{figure}
\centering
\definecolor{xdxdff}{rgb}{0.49019607843137253,0.49019607843137253,1.}
\begin{tikzpicture}[line cap=round,line join=round,x=1.0cm,y=1.0cm, scale=0.6]
\clip(-11,-1) rectangle (5,6.9);
\draw [dash pattern=on 2pt off 2pt,domain=-9.1:5.3] plot(\x,{(-0.-0.*\x)/8.});
\draw (0.,0.) -- (0.,6.68);
\draw [->] (0.,3.) -- (0.,4.);
\draw [domain=-9.1:5.3] plot(\x,{(--6.-0.*\x)/2.});
\draw [domain=-9.1:5.3] plot(\x,{(--9.696-0.*\x)/2.02});
\draw [->] (0.,4.8) -- (0.,5.8);
\draw(-4.3,2.68) circle (2.689014689435519cm);
\draw(-6.18,4.8) circle (4.800374985352707cm);
\begin{scriptsize}
\draw [fill=xdxdff] (0.,3.) circle (1.5pt);
\draw [fill=xdxdff] (0.,4.8) circle (1.5pt);
\draw [fill=xdxdff] (-3.813841500853314,5.32470223535797) circle (1.5pt);
\draw [fill=xdxdff] (-1.4151636471736069,5.383039047392259) circle (1.5pt);
\end{scriptsize}
\node (A) at (-0.3,2.6) {$i$};
\node (B) at (0.5,3.5) {$\tilde{u}$};
\node (C) at (-0.6,4.4) {$ie^{t_0}$};
\node (D) at (1,5.3) {$g_{t_0}(\tilde{u})$};
\node (E) at (4,4.3) {$g_{t_0}(H)$};
\node (F) at (4,2.5) {$H$};
\node (G) at (-4.1,5.8) {$\gamma'_ni$};
\node (G') at (-2.4,5.7) {$\gamma'_{n+1}i$};
\node (H) at (-7.5,3.7) {$\gamma'_nH$};
\node (I) at (-8.4,1.6) {$\gamma'_{n+1}H$};
\node (J) at (-4.2,-0.5) {$\gamma'_n\infty$};
\node (K) at (-6.4,-0.5) {$\gamma'_{n+1}\infty$};
\end{tikzpicture}
\caption{Horocycles $\gamma'_nH$ approaching $g_{t_0}(H)$}
\label{horapproach}
\end{figure}
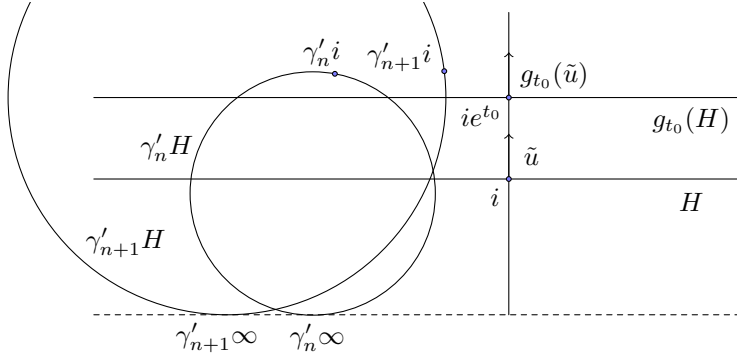

\noindent {\em Step 2.} We will show that 
$$B_{\gamma_n\infty}(i,\gamma_n i)=B_\infty(\gamma_n^{-1} i,i)\leq  b.$$
We have that
$$B_\infty(\gamma_n^{-1} i,i)=B_\infty(\gamma_n^{-1}i,\gamma_n^{-1}r(t_n))+B_\infty(\gamma_n^{-1}r(t_n),r(t_n))+B_\infty(r(t_n),i).$$
Computing these three terms gives
\begin{list}{\labelitemi}{\leftmargin=5pt}
 \item[(1)] $B_\infty(r(t_n),i)=-t_n$,
 \item[(2)] $B_\infty(\gamma_n^{-1}r(t_n),r(t_n))\leq d(\gamma_n^{-1}r(t_n),r(t_n))\leq b$, since both $\gamma_n^{-1}r(t_n)$ and $r(t_n)$ 
 belong to
 the closed geodesic $\alpha_n$ the length of which is bounded by $b$, and
 \item[(3)] $B_\infty(\gamma_n^{-1}i,\gamma_n^{-1}r(t_n))\leq d(\gamma_n^{-1}i,\gamma_n^{-1}r(t_n))=
 d(i,r(t_n))=t_n$.
\end{list}
Therefore
$$B_{\gamma_n\infty}(i,\gamma_n i)\leq -t_n+b+t_n=b.$$
Notice that, using the same proof, we obtain that for $k>0,\ B_{\gamma_n\infty}(i,\gamma_ni)\leq kb.$

\medskip
\noindent {\em Step 3.} We would like to show that there exists $C>0$ such that
$$B_{\gamma_n\infty}(i,\gamma_n i)\geq C.$$
In fact we will see that there is a positive $k$ such that for every $n$
$$B_{\gamma^k_n\infty }(i,\gamma^k_n i)\geq C.$$

As before, we consider the decomposition
$$B_\infty(\gamma_n^{-1}i,i)=B_\infty(\gamma_n^{-1}i,\gamma_n^{-1}r(t_n))+B_\infty(\gamma_n^{-1}r(t_n),r(t_n))+B_\infty(r(t_n),i),$$
and we will compute each of these three terms.

For the first term we
consider the geodesic ray $c_n$ going from $r(t_n)$ to $\gamma_n\infty$. We have that
$$B_{\gamma_n\infty}(i,r(t_n))=\lim_{t\to+\infty}d(i,c_n(t))-d(r(t_n),c_n(t)).$$
Take the geodesic path from $\gamma_n^+$ to $r(t_n)$ and the one from $r(t_n)$ to $i$. 
They form an angle greater or equal to $\pi/2$ (see Figure~\ref{angles}).Then the geodesic path
from $c_n(t)$ to $r(t_n)$ and the one from $r(t_n)$ to $i$ form an angle 
 greater than $\pi/2$ (see Figure~\ref{angles}).

Recall that hyperbolic triangles have the following property, that can be found in \cite{Sambusetti}: 
\begin{property}
 Let $\theta\in (0,\pi]$ and $d(\theta)=\log (2/(1-\cos\theta))$. If $ABC$ is a geodesic triangle in $\mathbb{H}$, such that 
 $\widehat{ACB}\geq \theta$, then 
 $$d(A,B)\geq d(A,C)+d(C,B)-d(\theta).$$
 In particular, if $\theta\geq \pi/2$, $d(A,B)\geq d(A,C)+d(C,B)-\log(2)$.
\end{property}

Therefore 
$$d(i,c_n(t))\geq d(i,r(t_n))+d(r(t_n),c_n(t))-\log (2)=t_n+t-\log (2),
$$
and this proves that
$$B_{\gamma_n\infty}(i,r(t_n))\geq t_n- \log(2).$$

\begin{figure}
\definecolor{uququq}{rgb}{0.25,0.25,0.25}
\definecolor{qqwuqq}{rgb}{0.13,0.13,0.13}
\definecolor{xdxdff}{rgb}{0.66,0.66,0.66}
\begin{tikzpicture}[line cap=round,line join=round,>=triangle 45,x=1.0cm,y=1.0cm, scale=0.5]
\clip(-9.68,-1.34) rectangle (6.34,8.26);
\draw [shift={(-0.01,4.98)},line width=0.4pt,color=qqwuqq,fill=qqwuqq,fill opacity=0.1] (0,0) -- (153.41:0.8) arc (153.41:270.17:0.8) -- cycle;
\draw [shift={(-0.01,4.98)},color=qqwuqq,fill=qqwuqq,fill opacity=0.15] (0,0) -- (167.53:0.5) arc (167.53:270.17:0.5) -- cycle;
\draw [shift={(-0.01,4.98)},line width=0.4pt,color=qqwuqq,fill=qqwuqq,fill opacity=0.2] (0,0) -- (-26.59:0.5) arc (-26.59:89.74:0.5) -- cycle;
\draw [shift={(-2.63,0)},line width=0.4pt]  plot[domain=0: pi,variable=\t]({1*5.63*cos(\t r)+0*5.63*sin(\t r)},{0*5.63*cos(\t r)+1*5.63*sin(\t r)});
\draw [shift={(-2.63,0)},line width=1.2pt]  plot[domain=1.09: pi,variable=\t]({1*5.63*cos(\t r)+0*5.63*sin(\t r)},{0*5.63*cos(\t r)+1*5.63*sin(\t r)});
\draw [line width=0.4pt] (0,0)-- (0,2);
\draw [dash pattern=on 2pt off 2pt,domain=-9.68:6.34] plot(\x,{(-0-0*\x)/8.26});
\draw [shift={(-1.14,0)},line width=0.4pt]  plot[domain=0: pi,variable=\t]({1*5.1*cos(\t r)+0*5.1*sin(\t r)},{0*5.1*cos(\t r)+1*5.1*sin(\t r)});
\draw [shift={(-1.13,0)},line width=1.2pt]  plot[domain=1.35:3.14,variable=\t]({1*5.11*cos(\t r)+0*5.11*sin(\t r)},{0*5.11*cos(\t r)+1*5.11*sin(\t r)});
\draw [line width=1.2pt] (0,2) -- (0,8.26);
\draw [shift={(-4.48,0)},dotted]  plot[domain=0.42:1.8,variable=\t]({1*4.91*cos(\t r)+0*4.91*sin(\t r)},{0*4.91*cos(\t r)+1*4.91*sin(\t r)});
\draw [shift={(-2.8,0)},dotted]  plot[domain=0.62:3.14,variable=\t]({1*3.44*cos(\t r)+0*3.44*sin(\t r)},{0*3.44*cos(\t r)+1*3.44*sin(\t r)});
\draw [shift={(-8.28,-0.11)},dotted]  plot[domain=0.29:0.71,variable=\t]({1*10.88*cos(\t r)+0*10.88*sin(\t r)},{0*10.88*cos(\t r)+1*10.88*sin(\t r)});
\draw [fill=xdxdff] (-0.01,4.98) circle (2pt);
\draw [fill=xdxdff] (-5.61,4.78) circle (2pt);
\draw[fill=xdxdff] (0,2) circle (2pt);
\draw [fill=xdxdff] (2.14,2.98) circle (2pt);
\draw [fill=xdxdff]  (0,6.94) circle (2pt);
\node (A) at (-8.2,-0.5) {\scriptsize $\gamma_n\infty$};
\node (B) at (-6.1,-0.5) {\scriptsize $\gamma^+_n$};
\node (C) at (-6.5,5) {\scriptsize $c_n(t)$};
\node (D) at (0.4,2) {\scriptsize$ i$};
\node (E) at (1.4,5.2) {\scriptsize$ r(t_n)$};
\node (F) at (0.8,7) {\scriptsize$ r(t)$};
\node (e) at (3.4,3.3) {\scriptsize$ \gamma_n^{-1}r(t_n)$};
\end{tikzpicture}
\caption{Angles in Step 3}
\label{angles}
\end{figure}
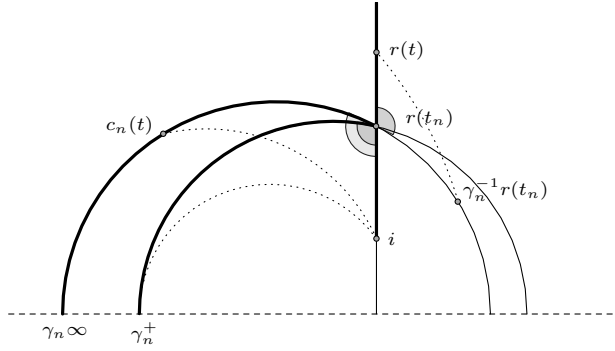

Now we will deal with the second term.

$$B_\infty( \gamma_n^{-1}r(t_n),r(t_n))=\lim_{t\to\infty}d(r(t),\gamma_n^{-1}r(t_n))-d(r(t),r(t_n)).$$
Remark that $d(r(t),r(t_n))=t-t_n$.

The angle formed by the geodesic path from $\gamma_n^{-1}r(t_n)$ to $r(t_n)$ and the one from $r(t_n)$ to $r(t)$ is greater or equal 
to $\pi/2$ (see Figure~\ref{angles}). Therefore,  we have
$$d(\gamma_n^{-1}r(t_n),r(t))\geq d(\gamma_n^{-1}r(t_n),r(t_n))+(t-t_n)- \log(2).$$
We also know that $d(\gamma_n^{-1}r(t_n),r(t_n))=length(\alpha_n)\geq a$, so we have that
$$d(\gamma_n^{-1}r(t_n),r(t))-(t-t_n)\geq a - \log(2).$$

As before, the third term is equal to $t_n$.

Putting everything together, we have that
$$B_{\gamma_n\infty}(i,\gamma_n i)=B_\infty(\gamma_n^{-1}i,i)\geq t_n- \log(2)+a- \log(2)-t_n=a-  2\log(2).$$

To obtain the desired conclusion, it would be enough to verify that
$$
a- 2\log(2)>0.
$$
This might not hold, but if we
replace $\gamma_n$ with $\gamma_n^k$ for a sufficiently large $k$, we will get
$$B_{\gamma_n^k\infty}(i,\gamma_n^k i)\geq ka- 2\log(2)>0.$$

An appropriate subsequence of the $\gamma_n^k$ satisfies conditions (i) and (ii).

\end{proof}

Notice that the Key Lemma also holds if $S$ is a Riemannian surface whose sectional curvature is less or equal than $-1$.

\medskip

Using Remark \ref{remark:geodesic_rays} and Lemma \ref{key_lemma}, we get the following corollary.

\begin{corollary} \label{dichotomy}
Let $S$ be a coarsely tame hyperbolic surface of bounded geometry. Then,
for any tangent vector $u \in T^1 S$, either $\overline{h_{\!_\R}\!(u)} = T^1 S$ or
there is a real number $t_0 > 0$ such that $g_{t_0} (u) \in \overline{h_{\!_\R}\! (u)}$.
\end{corollary}

 Notice that Corollary \ref{dichotomy} holds in the more general context where $S$ admits a pair of pants decomposition
such that any geodesic ray either stays in a compact region or intersects infinitely many closed geodesics whose length belongs
to some interval $[a,b]$, with $b\geq a >0$, depending on the ray. 
Surfaces for which there exists $b$ are called {\em weakly tame} in \cite{Sarig}.

Using this corollary, we retrieve a result due to S.Matsumoto (see \cite{Matsumoto}, \cite{Matsumoto2}) saying
that the horocycle flow on $T^1S$ admits no minimal sets.

To see this, let $S=\Gamma\backslash\H$ be a coarsely tame hyperbolic surface of bounded geometry,
and suppose that $\mathcal{C}\subset T^S$ is a minimal set. Since $S$ is non-compact, $C\neq T^1S$
(see \cite{Dal'Bo}).
Using Corollary \ref{dichotomy} we get $t_0$ such that $g_{t_0}(\mathcal{C})=\mathcal{C}$. On the other hand, if
$u\in \mathcal{C}$ and $\tilde u$ is a lift of $u$ to $T^1\H$, the point at infinity $\tilde u(+\infty)$ is
non-horocyclic. 
  For any $k\in\Z$, let $\zeta_k$ be the base point of $\tilde g_{kt_0}(\tilde u)$, where $\tilde g_\R$ denotes the geodesic 
flow. Notice that since $g_{kt_0}(u)\in \mathcal{C}$, there exist $(t_n)_{n\geq 0}$ in $\R$ and $(\gamma_n)_{n\geq 0}$ in $\Gamma$ 
such that $\gamma_n\tilde h_{t_n}(\tilde u)$ converges to $\tilde g_{kt_0}(u)$. This implies that 
$B_{\gamma_n(\tilde u(+\infty))}(\zeta_0,\gamma_n\zeta_0)=B_{\tilde u(+\infty)}(\gamma_n^{-1}\zeta_0,\zeta_0)$ converges to 
$B_{u(+\infty)}(\zeta_0,\zeta_k)=kt_0$. Hence $\Gamma_{\zeta_0}$ meets a horodisc centered at $\tilde u(+\infty)$--but this is not
the case, since $\tilde u(+\infty)$ is not horocyclic.

\section{Topology of leaves} \label{Sfol}

The purpose of this section is to show that there are in fact many foliations which are similar to the Hirsch foliation.
More precisely, we will prove Theorem \ref{reduction_theorem}, which has been stated in the Introduction.
\medskip

The proof of this theorem will make use of two preliminary results. The first one is about 
 simply connected leaves
and the second one is the construction of a certain type of uniformization of $(M,\cF)$. 
We also will make use of the following remark, which is in itself an
interesting fact about compact laminations by hyperbolic surfaces:

\begin{remark}

A compact lamination by hyperbolic surfaces has the property that no leaf has a cusp. 
This is because the injectivity radius of all leaves is bounded below by a positive constant. 
A more dynamical 
way to see it is the following:
If there were a leaf with a cusp, 
the 
horocycle flow on the unit tangent bundle would have a 
closed orbit $h$  
(a closed horocycle) and the geodesic flow
$g_t$ would have the property that 
the diameter of $g_t(h)$ converges to $0$. Taking a subsequence the 
set $g_{t_n}(h)$ would converge to a fixed point of the geodesic flow, which is a contradiction.
\end{remark}

\medskip


\begin{proposition} \label{vanishingcycles}
Let  $(M,\cF)$ be minimal compact lamination by hyperbolic surfaces.Then $\cF$ has no simply connected leaves if and only if there is 
a leaf that contains an essential loop with trivial holonomy.
\end{proposition}

\begin{proof} Assume that a compact lamination $(M,\cF)$, which is not necessarily minimal, has no simply connected leaves. 
A theorem due to D.B.A. Epstein, K.C. Millet and D. Tischler \cite{EpsteinMilletTischler} and independently to G. Hector \cite{Hector} states 
that the leaves with trivial holonomy form a residual set. Since any leaf in this residual set is not simply connected, it must contain an 
essential loop which has obviously trivial holonomy. 
Reciprocally, assume that there is a leaf $L$ that contains an essential loop $c_0:\mathbb{S}^1\to{L}$  without holonomy. By the standard
length-shortening arguments $c$ is freely homotopic to a closed geodesic in $L$, and we can assume that it is in fact a closed geodesic
parametrized by arc length by a periodic function
$c_0:\R\to{L}$  of period $s_0$. Since it has trivial holonomy and the metric varies continuously from leaf to leaf
in the $C^\infty$ topology (although $C^3$ is enough), Reeb's stability theorem implies that 
for all $\varepsilon>0$, there exist a small transversal $T$ through the point $t_0=c_0(0)$ and a smooth map
$\hat c:\R \times T \to{M}$ such that:

\begin{enumerate}
\item $\hat c(s,t_0)=c_0(s)$ for all $s\in\R$.
\item $\hat c(s+s_0,t)=\hat c(s,t)$ for all  $s\in\R$, $t\in T$. Let $c_t(s)=\hat c(s,t)$.
\item For all $t$ in $T$ the curve $c_t:\R\to{M}$ has image in a leaf  $L_t$ of the  lamination passing through
$t\in T$.
\item For all $t$ in  $T-\{t_0\}$ the curve $c_t:\R\to{M}$ is closed and there is a
geodesic $\alpha_t$ in $L_t$ such that $d(\alpha_t(s),c_t(s))<\varepsilon$ for $s\in [0,s_0]$.
\end{enumerate}

Let us join the points $\alpha_t(s_0)$ and $\alpha_t(0)$ by the geodesic segment $\beta_t$
in $L_t$ that realizes the distance between them. If $\epsilon$ is sufficiently small, the concatenation of
 $\alpha_t\!\mid_{[0,s_0]}$ and $\beta_t$ is freely homotopic to $c_t$. Since $L_t$ has negative curvature, this
 means that $c_t$ is essential in $L_t$. Finally, as $\cF$ is minimal, any leaf of $L'$ meets the image of $\hat c$ and hence 
 there exists $t \in T$ such that $L' = L_t$ is not simply connected. 
\end{proof}

By construction, the continuous map $\hat c:\R \times T \to{M}$ factors over a continuous extension $c : \mathbb{S}^1\times T\to M$ of the 
closed geodesic $c_0$ so that the loop $c_t = c \mid_{\mathbb{S}^1\times\{t\}}$ is essential in the leaf $L_t$ for each $t \in T$. In fact, using again 
Reeb's stability theorem as above, the inclusion of a $\varepsilon$-tubular neighborhood $A$ of $c_0$ into $M$ extends to a continuous map 
$c : A \times T\to M$ such that $c(A \times \{t\})$ is a $\varepsilon$-tubular neighborhood of the essential loop $c_t$ into the leaf $L_t$ for all $t \in T$. Note the 
persistence of the closed geodesic in the free homotopic class of $c_t$ when the hyperbolic metric varies in the 
transverse direction. In fact, by the normal hyperbolicity of the geodesic flow in the sense of M.W. Hirsch, C.C. Pugh and M. 
Shub \cite{HirschPughShub}, the closed geodesic varies continuously in the $C^\infty$ topology, and we can finally assume that 
every essential loop $c_t$ is in fact a simple closed geodesic.

\begin{definition} Let  $(M,\cF)$ be a compact minimal lamination by hyperbolic
 surfaces. Then a smooth map $\Phi:D\times T \to{M}$ is a {\em uniform systole cylinder cover} if:
 \begin{enumerate}

 \item $D$ is diffeomorphic to a cylinder $\mathbb{S}^1\times(0,1)$.
 \item $\Phi$ is a local homeomorphism  (which is not necessarily a covering map) and $\Phi(D\times T)={M}$. 
 \item For each $t\in T$ the map $\Phi_t:D\to{M}$, $x\mapsto{\Phi(x,t)}$,
  has image in 
  the leaf  $L_t$ of the  lamination passing through $\Phi(x_0,t)$, where $x_0$
  is a fixed point of $D$, and
$\Phi_t:D\to \Phi_t(D)\subset L_t$ is a local diffeomorphism.
\item Let $g^*$ be the continuous leafwise Riemannian metric on $D\times T$ which is the pull-back
metric, under $\Phi$, of the continuous metric $g$ on $M$ which renders all leaves of $\cF$ hyperbolic.
Under these two metrics restricted to the leaves, $\Phi_t$
is a local isometry between $D\times\{t\}$ and its image in $L_t$.
\item For each $t\in T$ the surface $D\times\{t\}$ with the metric induced by $g^*$ has a unique closed geodesic
$c_t$,  the systole of the cylinder, and it is a tubular neighborhood of this geodesic whose fiber has uniformly
bounded length.

 \end{enumerate}

\end{definition}

\begin{theorem}\label{Uniform_systole_covering}
  Let  $(M,\cF)$ be a compact minimal lamination by hyperbolic
 surfaces. If there is an essential loop without holonomy, then $\cF$ admits a uniform systole cylinder
 cover.
\end{theorem}

\begin{proof}

 According to the proof of Proposition~\ref{vanishingcycles} and the subsequent discussion, 
 we have a continuous map 
 $c:\mathbb{S}^1\times T\to M$,
 where each $c_t$ is a simple closed geodesic on a leaf $L_t$.

Let $\mathbb{D}$ be the Poincar\'e disk with the Poincar\'e metric and $F:\mathbb{D}\times T\to{M}$
be the global uniformization map (see \cite{Candel} and \cite{Verjovsky}), i.e. the continuous
map that satisfies the following conditions:
\begin{enumerate}
\item The restriction $F_t$ of $F$ to $\mathbb{D}\times \{t\}$ is a uniformization of $L_t$,
i.e. it is a covering map
which is a local isometry between the Poincar\'e disc and the hyperbolic leaf $L_t$.
\item $F(0,t)=c_t(1)$  and   $\frac{d~}{dz}(F(z,t))|_{z=0}=\dot{c}_t(1)$.
\end{enumerate}

Let $\tilde c_t$ be the lift of $c_t$ which passes through $0\in \mathbb{D}$, and
$\gamma_t\in PSL(2,\R)$ be the M\"obius transformation of $\mathbb{D}$ that has axis $\tilde{c}_t$ and such that
$c_t$ is the quotient of $\tilde c_t$ by the cyclic group generated by $\gamma_t$. Remark that $\gamma_t$
varies continuously with $t$.

We define the skew-cylinder
$$N=\mathbb{D}\times T/\sim$$
where $(z,t)\sim (\gamma_tz,t)$, with the induced Riemannian metric along the $z$-direction.

Then the uniformization map $F$ can be factorized through $N$
$$
\xymatrix{
& \mathbb{D} \times T   \ar[dr]_F \ar[rr]^\pi & & N \ar[dl]^\varphi  & \hspace{-2.8em} =\mathbb{D}\times T/\sim \\
 && M & &
}
$$
verifying
\begin{enumerate}
\item There exists a diffeomorphism $f:\mathbb{S}^1\times \R\times T\to N$ which
sends $\mathbb{S}^1\times\{0\}\times \{t\}$ onto the systole $c_t$.
\item The map $\varphi$ is a local homeomorphism; in particular it is an open map.
\item Restricted to $\mathbb{D}\times \{t\}$ and $\pi(\mathbb{D}\times \{t\})$ respectively,
both $\pi$ and $\varphi$ are locally isometric covering maps.
\end{enumerate}

For every $k\in \N$, take $N_k=f(\mathbb{S}^1\times (-k,k)\times T)$. Its image under $\varphi$
is open in $M$, and $\{\varphi(N_k):\ k\in \N\}$ is a cover of $M$. Since $M$ is compact, there exists
a $k_0$ for which $\varphi(N_{k_0})=M$, and therefore the restriction of $\varphi$ to
$N_{k_0}$ defines a uniform systole cylinder cover
$$
\Phi: D\times T= \mathbb{S}^1\times (-k_0,k_0) \times T \xrightarrow{f}   N_{k_0} \xrightarrow\varphi  M.
$$
\end{proof}

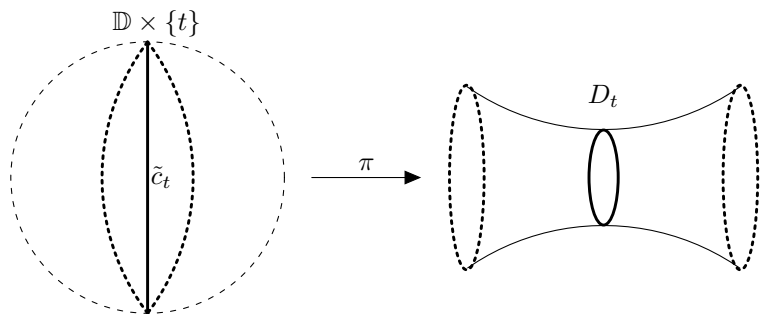
\begin{figure}
\begin{tikzpicture}[line cap=round,line join=round,>=triangle 45,x=1.0cm,y=1.0cm,scale=0.6]
\clip(-3,-3) rectangle (13.5,4);
\draw [dash pattern=on 2pt off 3pt] (0,0) circle (3cm);
\draw [line width=1.1pt] (0,3)-- (0,-3);
\draw [shift={(4,0)},line width=1.1pt,dotted]  plot[domain=2.5:3.79,variable=\t]({1*5*cos(\t r)+0*5*sin(\t r)},{0*5*cos(\t r)+1*5*sin(\t r)});
\draw [shift={(-4,0)},line width=1.1pt,dotted]  plot[domain=-0.64:0.64,variable=\t]({1*5*cos(\t r)+0*5*sin(\t r)},{0*5*cos(\t r)+1*5*sin(\t r)});
\draw [rotate around={90:(10,0)},line width=1.1pt] (10,0) ellipse (1.05cm and 0.32cm);
\draw [shift={(10,6.27)}] plot[domain=4.1:5.33,variable=\t]({1*5.21*cos(\t r)+0*5.21*sin(\t r)},{0*5.21*cos(\t r)+1*5.21*sin(\t r)});
\draw [shift={(10,-6.27)}] plot[domain=0.96:2.18,variable=\t]({1*5.21*cos(\t r)+0*5.21*sin(\t r)},{0*5.21*cos(\t r)+1*5.21*sin(\t r)});
\draw [rotate around={90:(7,0)},line width=1.1pt,dotted] (7,0) ellipse (2.04cm and 0.38cm);
\draw [rotate around={90:(13,0)},line width=1.1pt,dotted] (13,0) ellipse (2.04cm and 0.38cm);
\draw[->] (3.6,0)-- (6,0);
\node (A) at (0.2,3.4) {$\mathbb{D} \times \{t\}$}; 
\node (B) at (0.3,0) {$\tilde{c}_t$}; 
\node (C) at (4.8,0.3) {$\pi$}; 
\node(D) at (10,1.8){$D_t$};
\end{tikzpicture}
\caption{A systole cylinder}
\label{systole}
\end{figure}
 \bigskip

 We are now in a position to prove Theorem \ref{reduction_theorem}.

 \begin{proof}[Proof of Theorem \ref{reduction_theorem}] The equivalence $(i) \Leftrightarrow (ii)$ has been proved in 
 Proposition~\ref{vanishingcycles}, and now we will prove $(ii)\Rightarrow (iv)\Rightarrow (iii) \Rightarrow (ii)$.

  \medskip

  \noindent{$(i) \Rightarrow (iii)$}
   Since $\cF$ is minimal, it is clear that all its leaves are noncompact. We wish to prove that there are positive constants $a$ and $b$ such
   that if $L$ is a leaf of $\cF$ and $r$ is any geodesic ray going to infinity in $L$, then
   $r$ must intersect infinitely many closed geodesics with lengths between $a$ and $b$. This is
   equivalent to proving that any geodesic ray going to infinity in $L$ intersects one geodesic
   with length between $a$ and $b$.

  \medskip

   We start by constructing a uniform systole cylinder cover and later we proceed in three steps. Indeed, 
   under the hypotheses that $\cF$ is minimal and that it has a leaf containing an essential loop without holonomy,
   Theorem \ref{Uniform_systole_covering} gives us a uniform systole cylinder cover $\Phi:D\times T\to M$,
   where, for every $t\in T$, $D_t=D\times\{t\}$ is an (incomplete) hyperbolic cylinder which is a tubular
   neighborhood of its unique closed geodesic. The boundary components of $\partial D_t$ are curves
   of constant geodesic curvature
   greater than 0 and less than 1. When lifted to the hyperbolic plane, they give two hypercycles,
   one at each side of the lift of the systole of $D_t$, joining its points at infinity (see Figure~\ref{systole}).
   The lengths of the systole and each of the boundary components of $D_t$ are
   bounded by two positive constants
   $a$ and $b$ which can be chosen uniformly
   for all $t$.
   In fact, we will simply take $a$ to be the lower bound of
   the injectivity radius on $\cF$.

  \medskip

   \noindent
   {\em Step 1}. In the first step, each leaf $L$ is covered by the countable family of surfaces
   $$\mathcal{U}_L=\{\Phi(D_t):\ L=L_t\}.$$
   Each element $S\in \mathcal{U}_L$ is an open hyperbolic surface
   of finite genus,
   which is relatively compact in $L$.
   Since $\Phi |_{D_t}$ arises from the uniformization covering $F_t : \mathbb{D} \to L_t$, its boundary is a union of finitely many 
   piecewise smooth simple
   closed curves contained in the images of 
   boundary components of $D_t$.
    We will actually refer to each of these by the term `boundary component' of $S$  even if in general the connected components of the 
    boundary are non-simple curves composed of several of these `boundary components' .
    Notice that  even if some boundary components of $S$ can be inessential in $L$, there are always 
    essential boundary components of $S$ in $L$.
    
  \medskip

   \noindent
   {\em Step 2}. Let $r$ be a geodesic ray that goes to infinity on $L$. It must intersect infinitely many of the surfaces
   $S\in\mathcal{U}_L$. In this second step, will see that for infinitely many of them,
   there is at least one boundary component of $S$ whose algebraic intersection
   number with $r$ is not zero. To define the intersection number of a ray with a closed
   curve, we must remove from $L$ the starting point of the ray and consider the intersection number
   in this new noncompact surface. We proceed in two stages:

  \medskip

   \noindent
   {\em Step 2.1}. First, we fix a surface $S\in \mathcal{U}_L$ containing the starting point of $r$ and we assume that the intersection number 
   of $r$ with every 
   boundary component of $S$ is zero. Let $C$ be the first boundary component of $S$ that is intersected by $r$ with the orientation induced by 
   the orientation 
   of $S$. Consider a parametrization $r:[0,+\infty)\to L$  of the ray $r$. Since $r$ intersects $C$, 
   it does so a finite number of times $\{t_1,\ldots,t_{2n}\}$, having intersection number $+1$ for $n$ times (including $t_1$) and intersection 
   number $-1$ for 
   other $n$ times.
   
   There are two consecutive parameter $t_i$ and $t_{i+1}$ where the intersection numbers have opposite sign. 
   Then we substitute $r$ for the concatenation of $r\!\mid_{[0,t_i]}$, $\beta_i$ and 
   $r\!\mid_{[t_{i+1},+\infty)}$, where $\beta_i$ is a segment of $C$ that joins $r(t_i)$ and 
   $r(t_{i+1})$. We have thus reduced the number of intersection points of the ray $r$ with the boundary component $C$. 
   Arguing inductively, we obtain a new curve that does not cut $C$ but that still goes to infinity. 
   Until it meets another boundary component of $S$, the new curve remains in the closure of $S$. 
   Since we can do the same for the next boundary component of $S$ that is intersected by $r$, we can inductively construct a curve that goes to the same 
   point at infinity as $r$ remaining in the closure of $S$,  which is absurd.

  \medskip

   \noindent
   {\em Step 2.2}. In general, $L$ is covered by an increasing sequence of relatively compact open hyperbolic surfaces which are made 
    up of finitely many surfaces in $\mathcal{U}_L$. Repeating the same argument used for $S$ for each of those surfaces, we deduce that the algebraic
    intersection of $r$ with at least one boundary component is not zero for infinitely many surfaces in $\mathcal{U}_L$. 
    
  \medskip

   \noindent
   {\em Step 3}.  In the final step, we consider a surface $S \in\mathcal{U}_L$ such that: 
   
   \begin{enumerate}
   \item its boundary is far from the starting point of $r$ in the sense that no boundary component of $S$ meets an open ball $B$ centered at this point,     
   \item the intersection number of $r$ with a boundary component $C$ of $S$ is not zero. 
   \end{enumerate}
   Then the surface which is obtained by removing the starting point admits a complete Riemannian metric having the same closed geodesics of length $\leq b$
    than the original one outside of  the neighborhood $B$ of the new end. Then $C$ is freely homotopic to a closed geodesic $\alpha$,
   and the ray $r$ must intersect it since their intersection number is still nonzero. Furthermore, since the
   length of $C$ is between $a$ and $b$ and $\alpha$ cannot be longer than $C$, we have proved that $r$
   intersects a closed geodesic of length between $a$ and $b$, as desired.

   This shows that $L$ is coarsely tame with bounded geometry.

    \medskip
   \noindent{$(iii)\Rightarrow (ii)$}
   Any surface which is coarsely tame with bounded geometry is geometrically infinite, so in fact
   all leaves are geometrically infinite.

  \medskip

   \noindent{$(ii) \Rightarrow (i)$}
   Let $L$ be a leaf which is geometrically infinite and take $x\in L$. Its fundamental group
   $\pi_1(L,x)$ is not finitely generated,
   and its holonomy group $Hol(L,x)$ is a finitely generated quotient of   $\pi_1(L,x)$.
   The non-trivial kernel of $\pi_1(L,x)\twoheadrightarrow Hol(L,x)$ is precisely the set
   of homotopy classes of loops  based at $x$ without holonomy.

 \end{proof}

\bibliography{Hirsch_nov2015}{}
\bibliographystyle{plain}

\end{document}